\documentclass[11pt]{amsart}  
\usepackage{amssymb,amsmath,amsthm,amscd}

\usepackage{xcolor}

\usepackage{graphicx}
\usepackage{lscape}
\usepackage{hyperref}

\numberwithin{equation}{section}

\newtheoremstyle{fancy1}{10pt}{10pt}{\itshape}{12pt}{\textsc\bgroup}{.\egroup}{8pt}{
}
\newtheoremstyle{fancy2}{10pt}{10pt}{}{12pt}{\itshape}{.}{8pt}{ }

\theoremstyle{fancy1}

\newtheorem{lem}[equation]{Lemma}
\newtheorem{prop}[equation]{Proposition}

\newtheorem*{thm*}{Theorem}

\newtheorem{main}{Theorem}
\newtheorem*{main*}{Theorem}

\newtheorem*{cor*}{Corollary}
\newtheorem*{prop*}{Proposition}

\newtheorem*{problem*}{Problem}

\theoremstyle{fancy2}

\newtheorem{rem}[equation]{Remark}
\newtheorem*{rems*}{Remarks}

\newtheorem*{rem*}{Remark}

\newtheorem*{example*}{Example}

\newcommand{\cref}[1]{Corollary~\ref{#1}}

\newcommand{\pref}[1]{Proposition~\ref{#1}}

\newcommand{\e}{\epsilon}

\newcommand{\Sph}{\mathbb{S}}
\newcommand{\Disc}{\mathbb{D}}

\newcommand{\C}{{\mathbb{C}}}
\newcommand{\R}{{\mathbb{R}}}

\newcommand{\QH}{{\mathbb{H}}}

\newcommand{\fg}{{\mathfrak{g}}}
\newcommand{\fk}{{\mathfrak{k}}}
\newcommand{\fh}{{\mathfrak{h}}}
\newcommand{\fm}{{\mathfrak{m}}}
\newcommand{\fn}{{\mathfrak{n}}}

\newcommand{\fp}{{\mathfrak{p}}}

\newcommand{\pro}[2]{\langle #1 , #2 \rangle}

\def\con#1=#2(#3){#1 \equiv #2 \bmod{#3}}

\newcommand{\ml}{\langle}                     
\newcommand{\mr}{\rangle}                    

\newcommand{\tr}{\ensuremath{\operatorname{tr}}}
\newcommand{\diag}{\ensuremath{\operatorname{diag}}}

\renewcommand{\Im}{\ensuremath{\operatorname{Im}}}
\newcommand{\Ad}{\ensuremath{\operatorname{Ad}}}

\newcommand{\Ric}{\ensuremath{\operatorname{Ric}}}

\DeclareMathOperator{\trace}{trace}

\DeclareMathOperator{\Id}{Id}

\DeclareMathOperator{\Hess}{Hess}
\DeclareMathOperator{\Scal}{Scal}
\DeclareMathOperator{\Div}{div}

\newcommand{\bi}{\bar{i}}

\newcommand{\oi}{\overline{i}}
\newcommand{\oj}{\overline{j}}

\begin{document}

		\title{Initial value problems on cohomogeneity one manifolds, I}

		\author{Luigi Verdiani}
		\address{University of Firenze}
		\email{luigi.verdiani@unifi.it}
		\author{Wolfgang Ziller}
		\address{University of Pennsylvania}
		\email{wziller@math.upenn.edu}
		\thanks{ The first named author was supported by the PRIN 2022 project
"Real and Complex Manifolds: Geometry and holomorphic dynamics" (2022AP8HZ9) and the second named author was supported by an ROG grant from the University of Pennsylvania. Both authors would like to thank  the Max Planck Institute in Bonn for their support.}

\begin{abstract}
	We study  initial value problems for various geometric equations  on a cohomogeneity  manifold near a singular orbit. We show that  when prescribing the Ricci curvature, or finding solutions to the Einstein and soliton equations,  there exist solutions near the singular orbit, unique up to a finite number of constants. In part I we make a special assumption that significantly simplifies the proof, and will solve the general case in Part II.
\end{abstract}

		\maketitle

\smallskip

A well known  problem is that of prescribing the Ricci curvature of a metric, i.e., given a symmetric
bi-linear tensor T, solve the differential equation $\Ric(g) = T$ (or on a compact manifold $\Ric(g) =
cT$ for some constant $c$) for a metric $g$.  There exists a local
solution if $T$ is non-degenerate or has constant rank \cite{De,DG1,DG2,DY}, but  in some cases  solutions do not exist. The problem has been studied in the case of cohomogeneity one actions in \cite{DK,CD1,CD2,Pu1,Pu2,BK}. In the homogeneous case, where it is
an algebraic problem, see the survey \cite{ BP}, and \cite{PZ1,PZ2} for more recent results.  

\smallskip

For a cohomogeneity one manifold, a PDE becomes an ODE and it is hence easier to find examples, explicit or non-explicit, as well as classifications, in various geometric problems. Just to mention a few, see \cite{EW,DW1,DW2,PS1,PS2,BS,Bo,FH,Ch2,Wi,VZ2}.
In such results it is often important to first understand the initial value problem near a singular orbit, which is already non-trivial in general. 
  We will describe a method which can be used to solve such IVP's in many geometric problems. We will assume throughout the paper, that a singular orbit exists since otherwise the initial value problem becomes a regular ODE.

\smallskip

Before stating our result, let us recall the structure of a cohomogeneity one manifold near a collapsing orbit.  One starts with three Lie groups   $H\subset K\subset G$  with $K/H=\Sph^\ell$, $ \ell>0$. The action of $K$ on $\Sph^\ell$ extends to a linear action on $\Disc=\Disc^{{\ell}+1} \subset \mathbb{R}^{\ell+1}$ and thus $M=G\times_K \Disc$ is a homogeneous disc bundle with boundary $G\times_K\partial\, \Disc=G\times_KK/H=G/H$ a principal orbit.     
 By the slice theorem, any cohomogeneity one manifold has this structure near a singular orbit. 

\smallskip To be more explicit about the metric, define a splitting of the Lie algebras $\fg=\fh\oplus\fp\oplus\fm$, orthogonal with respect to a background metric $Q$ on $\fg$, where $\fh\oplus\fp$ is the Lie algebra $\fk$. In  Part I of this paper we assume that:
\begin{equation*}
	(*)\quad \text{ There are no $\Ad_H$ irreducible representations        in  }    \fm              \text{ which are equivalent  to one in } \fp.
\end{equation*}
This assumption guarantees that $\fp$ and $\fm$ are orthogonal with respect to any $\Ad_H$ invariant bi-linear form, in particular the metric, the tensor $T$ and the Ricci tensor. This will significantly simplify the proofs and we will remove this assumption in Part II. Notice that since the equation is second order, we expect to get a solution after choosing the metric and the second fundamental form of the singular orbit arbitrarily.
\begin{main}
	Let  $M=G\times_K \Disc$ be a homogeneous disc bundle, describing the cohomogeneity one manifold near the singular orbit $G/K$ and satisfying condition (*). Let $T$ be a smooth symmetric bi-linear form on $M$.  Given the metric and the second fundamental form of the singular orbit, there exists a smooth metric $g$ in a neighborhood of $G/K$ with $\Ric(g)=T$, when restricted to the principal orbits. The solution is unique, up to finitely many arbitrarily chosen  constants. 
\end{main}
We will also  show that one can  prescribe $T(\dot c,\dot c)$
 by reparametrizing the geodesic, see Section 6.2. But in general  $T(\dot c,X)$ with $X$ tangent to the regular orbits can  not be prescribed, unless it is forced to be zero by $\Ad_H$ invariance, i.e. $\fp_0=\fm_0=\{0\}$. For special $T$, e.g., when $T$ is diagonal, the full tensor can be prescribed.  If $T$ is analytic, the solution is analytic as well.

\smallskip

We now describe the method of proof, and the values of the free parameters. We fix a smooth curve $c(t)$ transverse to all orbits and choose a metric of the form $g=dt^2+g_t$ where $g_t$ is a one parameter family of homogeneous metrics on the regular hypersurface orbits $G\cdot c(t)=G/H$.  The metrics $g_t$ are determined by a positive definite matrix of smooth functions $g_{ij}(t),\ t>0$. These functions need to satisfy an ODE which blows up at the singular orbit. The ODE is of regular singular type and hence  does not necessarily have a solution. We will show that in our case there exists a  formal power series solution. If $T$ is smooth, there exists a nearby smooth solution to $\Ric(g)=T$ by \cite{Ma1}. If $T$ is analytic,  the power series converges, see \cite{Ma2,GG}, and hence the solution is analytic as well.

  One of the problems is that one first needs to determine when the metric (and the tensor) on the regular part extends smoothly to the singular orbit. In \cite{VZ1} it was shown that this can be done in terms of a set of even functions. More precisely, there exist  $a_{ij}^k\in \mathbb{R}$ and integers $ d_k\ge 0$  such that the metric $g_{ij}$ has a smooth extension to the singular orbit if and only if
$$\sum_{i,j} a_{ij}^k\,g_{ij}(t)=t^{d_k}\phi_k(t^2)\quad \text{ for } k=1,\cdots, N,  $$
where $\phi_1,\cdots,\phi_{N}$ are smooth even functions. The numbers  $a_{ij}^k $ and the integers $d_k$  are determined by  the structure constants of the Lie group $G$. This over determined  system of $N$  equations can  be solved for the coefficients  of the metric, thus describing $g_{ij}$ in terms of $r$ even functions, where $r$ is the number of distinct non-zero entries in $g_{ij}$. The advantage of this method is that  smoothness of the metric is now automatically built in.

By substituting these expressions for the metric into  the differential equations, we   will  show that one can  solve for the second derivatives of $\phi_i$, thus obtaining a system of the form $\phi_i''=F_i(\phi_j,\phi_j',T),\ i=1,\cdots,N$. Since the ODE is of regular singular type,
the smoothness of the   functions $F_i$  is sometimes obstructed unless  some of the initial conditions $\phi_i(0)$  satisfy certain (linear) {\it compatibility conditions}. Once these are satisfied, one  obtains, by differentiating $F_i$, a formal power series for the metric.

We will see that  the same proof   can in fact be applied to many geometric ODE's on a cohomogeneity one manifold. This holds e.g., for Einstein, K\"ahler Einstein, soliton or more generally quasi-Einstein metrics. In this case the second Bianchi identity implies that, if there exists a singular orbit,  the regular orbits are orthogonal to $\dot c$, and that $\Ric(\dot c, \dot c)$ is constant. Thus we have:
\begin{main}
	Let  $M=G\times_K \Disc$ be a homogeneous disc bundle satisfying condition (*). Given the metric and the second fundamental form of the singular orbit,  there exists an analytic solution to the Einstein equation, or the soliton equation,  in a neighborhood of $G/K$. The solution is unique, up to finitely many arbitrarily chosen  constants. 
\end{main}

\smallskip

To be more explicit about the compatibility conditions, and the choice of free parameters,  we expand the metric
\begin{equation*}
	g_{|\fm}=A_0+tA_1  +t^2 A(t),\qquad g_{|\fp}=t^2\Id+t^4B(t),\qquad g_{|\fp\fm}=t^2C_0 + t^3C(t)
\end{equation*}
where $A_0,A_1$ and $C_0$  are constant matrices and  $A,B,C$ are smooth matrix functions of  $t$.  Smoothness may imply further restriction on $A(t),B(t)$ and $C(t)$, see the discussion above. Notice though that under the assumption  (*) we have $g_{|\fp\fm}=0$. 
 Here,  $A_0$ and $A_0^{-1}A_1$ are the metric and the second fundamental form of the singular orbit $G/K$. Thus we want to choose $A_0$ and $A_1$ arbitrarily.

 Under the action of $\Ad_K$ we can split $\fm$ into irreducible modules $\fm_i$ and a trivial module $\fm_0$. Without assuming that $g_{|\fp\fm}=0$,  the following holds.
\begin{main}\label{compatibility}
	Let  $M=G\times_K \Disc$ be a homogeneous disc bundle. Then  we have the following properties for any cohomogeneity one metric on $M$.
	\begin{itemize}
\item[(a)] If $\fm_i\subset\fm$ is an irreducible $K$ module, then  
		$$
	\tr \Ric_M(0)_{|\fm_i}=		-(\dim(\fp) +1)\tr A(0)_{|\fm_i} +f_1(A_0,A_1,C_0)$$
\item[(b)]  If $\fm_i$ and $\fm_j$ are equivalent $K$ modules and  $u\to \bar u$ is an explicit equivalence , then 
$$
\sum_u \Ric_M(0)_{u \bar u}=(\dim(\fp) +1)\sum_u A(0)_{u \bar u}
+f_2(A_0,A_1,C_0)
$$
\vspace{2pt}
\item[(c)] 	 If $\fm_0\subset\fm$ is a trivial $K$ module, then  
			$$
		 \Ric_M(0)_{|\fm_0}=	- ( \dim(\fp)+1)\, A(0)_{|\fm_0}$$
\item[(d)] On the module $\fp$ we have
$$
\lim_{t\to 0}\frac1{t^2}\tr\Ric(0)_{|\fp}= -3  \tr B(0)+f_3(A_0,A_1,C_0, A(0))
$$
	\end{itemize}
\end{main}
\noindent where $f_i$ are explicit polynomials.  In part (b), we need to consider 1,2,or 3 equivalences, depending on whether the representation is orthogonal, complex or symplectic. Notice that if  $\fm_i$ and $\fm_j$ are non-equivalent $K$ modules, then $ \Ric_M(0)(\fm_i,\fm_j)=0$ by Schur's Lemma. We also observe that  $A(0)$ and $B(0)$ in (a)-(d) are never forced to be identically $0$ by smoothness.

We will show, that in order to satisfy the compatibility conditions, we
must be able to prescribe $\Ric_M(0)_{|\fm_i}$, $\Ric_M(0)(\fm_i,\fm_j)$ and  $\lim_{t\to 0}\frac1{t^2}\Ric(0)_{|\fp}$ arbitrarily, after choosing $A_0$ and $A_1$. Since they are $K$ invariant, this is equivalent to prescribing their trace. Thus if we want to solve $\Ric(g)=T$ (or $\Ric(g)=\lambda g$), Theorem C states that the values of $\tr A(0)_{|\fm_i},\ A(0)_{|\fm_0} $ and $\tr B(0)$ for the metric $g$ are  determined by $T(0)$ (resp. $A_0$ and $\lambda$). After these are thus chosen, we obtain a formal power series for the metric. If $g_{|\fp\fm}\ne 0$, there are further compatibility conditions. In Part II we will show that these can be satisfied as well.

This implies that we  have a simple description of the free parameters in Theorem A and B. The trace free part of $A(0)$ and $B(0)$, if non-zero, as well as  $\phi_i(0)$ with $d_k>2$, are free parameters. If $\fm_i$ and $\fm_j$ are equivalent $K$ representations, all entries of $A_{u \bar u}$, except for one, are free parameters, and if  in-equivalent, all possible entries of the metric are free parameters. Thus we can immediately read off the number of free parameters, and their values in terms of the coefficients of the metric. In Section  5 we give some examples that illustrate our methods.

\smallskip

 Surprisingly, the contribution from $A(0)$ and $B(0)$ in Theorem C all come from the second fundamental form, and not the curvature of $G/H$.  Proving the latter is in fact the main difficulty in our proof of Theorem A and B.

\smallskip

We remark that all of our results hold more generally if   $G/H$ is a homogeneous space for which  $g(\fp,\fm)= 0$ implies $\Ric_{G/H}(\fp,\fm)=0$. See \cite{KPS} for a large collection of such examples, including many where $\fm_0\ne $\{0\}$\} $.

\smallskip

After one solves the initial value problem, the next question is of course weather there exist examples or obstructions to complete solutions on a cohomogeneity one manifolds. This is a rather delicate problem, 
 see e.g., \cite{PS1,DW1,DG1,DG2, Ba} or \cite{Bo, FH} in the compact case.

\smallskip

This initial value problem for Einstein metrics was studied in
\cite{EW} (and for soliton metrics in \cite{Buz}) under the assumptions (*) and  $\fm_0=\{0\}$ as well as that $G$ is compact. It  requires a significant amount of representation theory in order to apply it in concrete examples, whereas in our case this is not necessary. In  addition, smoothness conditions at the singular orbit are expressed in a more indirect way, and are thus difficult to determine explicitly in terms of the metric functions, see e.g., \cite{Ch2} Proposition 2.4. Notice that this result follows  immediately  from Theorem C.

 \smallskip

The paper is organized as follows. In Section 1, we recall the structure of a cohomogeneity one manifold, describe the metric in terms of even functions and discuss the smoothness conditions.  
 In Sections 2  we outline our strategy of proof, and discuss what needs to be proved in order to show that the compatibility conditions can be satisfied. In Section 3 we compute the contribution of the second fundamental form to the compatibility condition, and in Section 4 that of the intrinsic curvature of the regular orbit. In Section 5 we illustrate our method with some examples  and in Section 6 discuss further applications.

\section{\bf Preliminaries}

\bigskip

For a general reference for this Section see, e.g.,  \cite{AA,AB}. 
A  non-compact cohomogeneity one manifold is given by a homogeneous vector bundle
and a  compact and simply connected one by  the union of two
homogeneous disc bundles. We can restrict ourselves to only one such bundle. Let    $H,\, K ,\, G
$ be Lie groups with inclusions $H\subset K \subset G$ such that $H,K$ are compact and
$K/H=\Sph^{\ell},\ \ell>0$. The transitive action of $K$ on
$\Sph^{\ell}$ extends (up to conjugacy) to a unique linear action on the vector space $V=\R^{{\ell}+1} $.
We can thus define the homogeneous vector bundle
$M=G\times_{K}V$ and $G$  acts on $M$  via left action in the first component. This action has principal isotropy group $H$, and singular isotropy group $K$.
A neighborhood of the singular orbit is given by $G\times_{K}\Disc$ where   $\Disc\subset V$ is a disc centered at  $p_0=0\in V$.

\smallskip

We choose a vector $e_0\in V$ and 
consider the straight line $c(t)=te_0$.
At the regular points $c(t)$, i.e., $t>0$,
the isotropy is constant  equal to  $H$ and at $c(0)=0$ it is equal to $K$.
We fix an $\Ad_H$ invariant splitting $\fg=\fh\oplus\fn$ and identify the tangent space $T_{c(t)} G/H=\dot{c}^\perp\subset T_{c(t)}M$  with $\fn$
via action fields: $X\in\fn\to
X^*(c(t))$. The stabilizer group $H$ acts on $\fn$ via the adjoint representation
and a $G$ invariant metric on $G/H$ is described by an $\Ad_H$
invariant inner product on $\fn$. On the regular part, the metric along $c$
is thus given  by $g=dt^2+g_t$, with $g_t$ a one parameter family of $\Ad_H$ invariant inner products on the vector space $\fn$, depending smoothly on $t$. 

We choose a left invariant metric  on $G$, right invariant under $K$, determined by the inner product $Q$ on $\fg$. Hence $\Ad_K$ acts by isometries in $Q$.
We then choose a $Q$ orthogonal $\Ad_H$ invariant  splittings
$$
\fg=\fh\oplus\fn \text{ and } \fn=\fn_0\oplus\fn_1\oplus\ldots\oplus\fn_r.
$$
where $\Ad_H$ acts trivially on $\fn_0$ and irreducibly on $\fn_i$ for $ i>0$, and 
 an $\Ad_K$ invariant $Q$ orthogonal complement to $\fk\subset\fg$:
$$\fg=\fk\oplus \fm  , \quad  \fk=\fh\oplus \fp \ \text{ and thus }\ \fn=\fp \oplus \fm . $$
Here $\fm$ can be viewed as the tangent space to the singular orbit $G/K$ at
$p_0=c(0)$ and $\fp$ as the tangent space of the sphere $K/H\subset V$.
Furthermore, let
$$
\fm=\fm_0\oplus\fm_1\oplus,\cdots,\oplus\fm_k 
$$ 
be  a $Q$ orthogonal decomposition,
where $\fm_0$ is a trivial $\Ad_K$ module and  $\fm_i,i>0$ are $\Ad_K$ irreducible. We can assume that $\fn_j\subset\fp$ or $\fn_j\subset\fm_i$ for some $i$.   Finally,  choose a $Q$ orthogonal $\Ad_H$ invariant decompositions for $\fp$:
$$
\fp=\fp_0\oplus\fp_1\oplus\fp_2 
$$
where $\Ad_H$ acts trivially on $\fp_0$ and irreducible on $\fp_1$ and $\fp_2$. Depending on the sphere $K/H$, there may be one, two or three such modules.

Since $K$ preserves the slice $V$, and acts linearly on it,
we have an embedding $\fp\subset V$ defined by $X^*(c(t))=tX$ for $X\in\fp$.    In this language,    $V\simeq \dot c(0)\oplus \fp$. 
Notice that, since $K$ acts transitively on the unit sphere in $V$,  a $K$  invariant inner product on $V$ is determined  up to a multiple. Since for any $G$ invariant metric we fix a geodesic $c$, which we assume is parameterized by arc length, this uniquely  determines the inner product on $V$, which we denote by $g_0$. 
 Via the inclusion $\fp\subset V$, we will also use this inner product $g_0$ on $\fp$. Notice though that $g_0$ is a Euclidean inner product on $V$, but does not necessarily agree on $\fp$ with the restriction of a bi-invariant metric of $K$ (see [GZ], Lemma 2.5).

 We identify the homogeneous metrics $g_t$ with  an endomorphism $P_t$ in terms of action fields:
 \begin{equation}\label{metric}
 	P_t\colon\fn\to\fn\ \text{ where } \ g_t(X^*,Y^*)_{c(t)}=Q(P_tX,Y), \ \text{ for } \ X,Y\in\fn.
 \end{equation}

These endomorphisms $P_t$ commute with the action of $\Ad_H$ since the metric and $Q$ are  $\Ad_H$ invariant. By Schur's Lemma we have ${P_t}_{|\fn_i}=a(t) \Id_{|\fn_i}$ for $i>0$ and some positive function $a$, whereas on $\fn_0$ it is an arbitrary invertible endomorphism. Furthermore,  $\fn_i$ and $\fn_j$ are orthogonal if the representations of $\Ad_H$ are non-equivalent. If $\fn_i$ and $\fn_j$ are two equivalent modules, inner products between them are described by $\Ad_H$ equivariant intertwining operators $f\colon  \fn_i\to\fn_j$, which can be chosen to be $Q$ isometries. If the representations are orthogonal, then $f$ is unique and if $e_i$ is a $Q$ orthonormal basis of $\fn_i$, we let $e_{\oi}=f(e_i) $ be a $Q$ orthonormal basis of $\fn_j$. Then $Q(P_t(e_i),e_{\oj})=b(t)\,\delta_{ij}$ for some function $b\ne 0$.   If the representations are  complex, then $\fn_i$ has a $Q$ orthogonal complex structure $J\colon \fn_i  \to\fn_i$ and we choose a complex $Q$ orthonormal basis $\{e_i,J(e_i)\}$ of $\fn_i$ and the $Q$ orthonormal basis $\{e_{\oi}=f(e_i),\ J(e_{\oi})=f(J(e_i))\}$ of $\fn_j$. Then 
$$
Q(P_t(e_i),e_{\oj})=Q(P_t(J(e_i),J(e_{\oj})=c(t)\,\delta_{ij}, \ Q(P_t(e_i),J(e_{\oj})) =- Q(P_t(J(e_{i})),e_{\oj})=d(t)\,\delta_{ij}
$$
for some functions $c,d$.
Similarly, if $\fn_i$ is symplectic, we have three anticommuting complex structures on $\fn_i$ and corresponding inner products, as above, in terms of four functions. These are in fact all non-vanishing component of the metric.

Choosing a $Q$ orthonormal basis $e_i$ of $\fn$, the functions $g_{ij}=Q(P_t(e_i),e_j)$ define the metric on the regular part. At $t=0$ they satisfy certain smoothness conditions which guarantee that the metric $P_t$ extends smoothly to the singular orbit. These are somewhat complicated and were determined in \cite{VZ1}. For this, one chooses  closed one parameter groups 
$L_i=\exp(\theta v_i)\subset K$, one for each module $\fp_i$, i.e., $v_i\in\fp_i$. Then $V$ and $\fm$  decompose into 2 dimensional invariant subspaces $\ell_j $ isomorphic to $\C$, on which $L$ acts by multiplication with $e^{i\,m_j\theta}$. 
  Each such decomposition under $L_i$ will
give  constraints for the components of $P$  that have the form
\begin{equation}\label{smooth}
\sum a_{ij}^k\,g_{ij}=t^{d_k}\,\phi_k,\quad k=1,\cdots, N
\end{equation}
where $\phi_k$ is a smooth even function of $t$ and $d_k\geq 0$  an integer. The entries $a_{ij}^k$ are determined by the structure constants of $\fg$ and $d_k$ is described in terms of the above values of $m_j$ as in Table C in \cite{VZ1}.  
We may collect all the corresponding constraints and we get an over determined system of equations in terms of the even functions $\phi_i$. These functions  satisfy linear relations which enables one to express $g_{ij}$ in terms of a subset of $r$ even functions, where $r$ is the number of distinct  non-zero elements in $g_{ij},i\le j$.   See \cite{VZ1} Example 3 for a typical illustration.        Any symmetric bi-linear form $T$, in particular the Ricci tensor, satisfies the same conditions, with the same constants $a_{ij}^k$, except that possibly the exponents $d_k$ can differ by two for the mixed terms $T(\fp,\fm)$. Since the condition (*) implies that these mixed terms are $0$, this difference will be discussed in Part II.

For  $X\in\fp$, there is no need to take linear combinations. If $|X|_{g_0}=1$, smoothness simply  states that $|X|^2=t^2+t^4\phi$ for some even function $\phi$.

The formula for the Ricci tensor of $M$, restricted to the regular orbits, can be  obtained by the Gauss equation for the regular hypersurfaces:
\begin{equation}\label{Gauss}\Ric_M(X,Y) =\Ric_{G/H}(X,Y)+
Q(LX,Y)
\end{equation} where 
\begin{equation}\label{shape}
	L= -\frac 14 \tr(P^{-1}P') P' +\frac 12 P' P^{-1} P'-\frac 12 P''\ ,
\end{equation}
Note that if we regard $\Ric$ as endomorphisms, we need to replace  $\Ric$ by  $P^{-1}\Ric$. This needs to be taken into account if we compare the above  formulas with those in \cite{EW}   and in some other papers.  In our notation $P$ and $L$ are endomorphisms, and $P^{-1}P'$ is the shape operator. Sometimes we may regard $L$ as a bi-linear form, i.e., $L(X,Y)=Q(LX,Y)$, which will be clear from context.

 The intrinsic curvature of the homogeneous space is more complicated and presents the main difficulty in solving the initial value problem. One has the following formula for the Ricci tensor of a homogeneous space (see e.g., \cite{Be}):

\begin{align*}
\Ric_{G/H}(X,Y)=&-\frac12\sum_{i} g([X_i,X]_\fn,[X_i,Y]_\fn)+\frac14\sum_{i,j} g([X_i,X_j]_\fn,X)g([X_i,X_j]_\fn,Y)\\
&-\frac12 B(X,Y)-g( [Z,X],Y)
\end{align*}
where $Z=\sum_i U(X_i,X_i)$ with $2g(U(X,Y),Z)=g([Z,X],Y)+g([Z,Y],X)$. Furthermore, $B$ is the Killing form of $\fg$ and $X_i$ is a $g$-orthonormal basis of $\fn$.  The last term involving $Z$ vanishes when the group $G$ is unimodular (e.g., when $G$ is compact). But we will see in Section 6 that  it does not effect the compatibility conditions. For simplicity, we postpone this (easy) discussion, and thus assume from now on that $Z=0$.

We choose a $Q$ orthonormal basis $e_u$ of $\fn$ such that it is adapted to the decomposition of $\fp\oplus\fm$, the decomposition of $\fm=\fm_0\oplus\fm_1\oplus\cdots\oplus \fm_r  $ and furthermore to the decomposition of $\fn$ under $\Ad_H$. We choose  index sets $I$ and $J$ such that
$$
e_u\in \fp \text{ if } u\in I, \quad \text{and} \quad e_u\in \fm \text{ if } u\in J.
$$
Thus we have in the inner product $Q$:

\begin{align*}
\Ric_{G/H}(X,Y) 
=& -\frac12\sum_{r,s} P^{-1}_{rs}\ Q(P[e_r,X]_\fn,[e_s,Y]_\fn)   \\
& +\frac14\sum_{r,s,k,l}P^{-1}_{rs}P^{-1}_{kl} \ Q(P[e_r,e_k]_\fn,X)Q(P[e_s,e_l]_\fn,Y) -\frac12 B(X,Y).
\end{align*}

In this formula, $P_{rs}=Q(Pe_r,e_s)$. Now define the structure constants
$$
\Gamma_{ij}^u=Q([e_i,e_j],e_u) \ \text{ or equivalently }\  [e_i,e_j]_\fn=\sum_t \Gamma_{ij}^t e_t\ .
$$

\smallskip

$\Ad_K$ invariance implies the following symmetry properties for the structure constants:
\begin{align}\label{skew}
	\Gamma_{rs}^k = - \Gamma_{sr}^k ,\quad &\text{ for all } r,s,k \nonumber  \\ 
	\Gamma_{rs}^k = -\Gamma_{rk}^s ,\quad &\text{ for all } r,s,k \in I  \\
	\Gamma_{rs}^k = -\Gamma_{rk}^s ,\quad &\text{ for all } r\in I \text{ and } s,k \in J  \nonumber \\
	\Gamma_{rs}^k =0,\quad &\text{ if two indices lie in $I$ and the third in $J$  }  \nonumber
\end{align}
These are in fact all the symmetry properties that will be needed, in particular $G$ does not have to be compact.
A straightforward computation shows that  above formula turns into:
\begin{equation}\label{ric}
\Ric_{G/H}(e_u,e_v) =  -\frac12\sum_{r,s,p,q}\Gamma_{rp}^u\Gamma_{sq}^v P^{-1}_{rs} P_{pq}
+\frac14\sum_{r,s,k,l,p,q} \Gamma_{rk}^p\Gamma_{sl}^q  P^{-1}_{rs}P^{-1}_{kl}  P_{pu}P_{qv} -\frac12 B(e_u,e_v)
\end{equation}
for $e_u,e_v\in\fn$.
For example, if the metric is diagonal, i.e. $P_{rs}=x_s\delta_{rs}$, then
\begin{equation}\label{RicciDiagonal}
\Ric(e_u,e_v)=\sum_{r,k}\frac{x_ux_v-2x_k^2}{4x_rx_k}\Gamma^u_{rk}\Gamma^v_{rk}-\frac12 B(e_u,e_v).
\end{equation}
 Notice the Ricci curvature does not have to be diagonal, even if the metric is.

 For completeness, we also give the remaining components of the Ricci tensor:
 In arc length parametrization  we have

\begin{align}\label{Riccic}
	\Ric_M(\dot c,\dot c)&=  
	\tfrac{1}{4}\tr(P^{-1}P'P^{-1}P')-\tfrac{1}{2}\tr(P^{-1}P'')\\ \nonumber
	\Ric_M(e_u,\dot c)&=\frac12\sum_{k,s}\Gamma_{uk}^s (P^{-1}P')_{ks}
	-\frac12 \sum_{k,s}\Gamma_{ks}^s(P^{-1}P')_{ku}\\ \nonumber
	 -g( [Z,e_u],e_v)&=- \sum_{k,\ell,s,t}\Gamma_{ks}^s\Gamma_{\ell u}^tP^{-1}_{k\ell}P_{tv}\nonumber
\end{align}

\smallskip

\section{\bf Compatibility Conditions}
In this section we do not assume that $C=0$, and  hence the results can be used in Part II.

\noindent We want to solve the equation $\Ric(g)=T$ for a given smooth tensor $T$ (or the Einstein equation $\Ric(g)=\lambda g$). As we will see, $g$ is already uniquely determined, up to a finite collection of constants, if we prescribe only $T_{|\fn}$, i.e. the tensor $T$  tangent to the orbits.  For an Einstein metric (or soliton metric) the second Bianchi identity implies that on the regular orbits $\Ric(\dot c, \dot c)$ is constant and $\Ric(\dot c, \fn)=0$. Thus this is sufficient to prove the existence of an Einstein metric.     We prefer to use the language of prescribing the Ricci tensor, since it  illustrates the proof more clearly.
\smallskip

\smallskip
Recall that the metric is smooth iff
\begin{equation}\label{Metricsmooth}
	\sum_{i,j} a_{ij}^k\,g_{ij}(t)=t^{d_k}\phi_k(t^2)\quad \text{ for } k=1,\cdots, N, 
\end{equation}
where $\phi_1,\cdots,\phi_{r}$ are smooth even functions, and $ a_{ij}^k$ and $d_k$ do not depend on the metric.   The Ricci curvature, and the tensor $T$, satisfies the same smoothness conditions:
\begin{equation}\label{Ricsmooth}
	\sum_{i,j} a_{ij}^k\,\Ric_{ij}(t)=t^{d_k}\phi_k(t^2)\quad \text{ for } k=1,\cdots, N,  
\end{equation}
with the same coefficients $ a_{ij}^k$ and $d_k$. But if  $C\ne 0$, the integers $d_k$ may be two less than the one for the metric.

 We will show that it is possible to solve the equations 
\begin{equation}
\label{sysric}
\Ric(g)_{|\fn}=T_{|\fn}
\end{equation}
if and only  if the values of  $\phi_k(0)$ satisfy certain linear relations.
We will  call these the compatibility conditions. The proof will show that this linear system admits solutions, but some of the values of $\phi_k(0)$ may be chosen arbitrarily.

The Ricci curvature of a cohomogeneity one manifold is determined by the formulas \eqref{Gauss},\eqref{shape} and\eqref{ric}. They show that
it depends linearly on $P''$. Thus  in $\sum a_{ij}^k\,\Ric_{ij}$, the term $\phi_k''$ is the only second order derivative appearing, with a nonzero coefficient. Our strategy is then to solve the equations
\begin{equation}\label{newequations}
	\sum_{i,j} a_{ij}^k\,\Ric_{ij}(t)=\sum_{i,j} a_{ij}^k\,T_{ij}(t) 
\end{equation} for $\phi_k''$, in which case  \eqref{sysric} is equivalent to a system:
\begin{equation}\label{newequations}
\phi_k''=F_k(\phi_j,\phi_j',T),\ k=1,\cdots, N.
\end{equation}
We will show that the functions  $F_k$ are smooth  iff certain linear compatibility conditions among the initial conditions $\phi_k(0)$  hold. 
  By differentiating the equations \eqref{newequations} we find a formal power series for the functions $\phi_k, k=1,\cdots N$ and hence for the metric $g_{ij}$. If $T$ is smooth,  Malgrange's theorem \cite{Ma1} implies that there exists a nearby smooth solution. If $T$ is analytic, the power series converges, see  \cite{Ma2,GG},     and hence we obtain an analytic solution near the singular orbit. Notice that $\phi_k'(0)=0$, and hence prescribing $\phi_k(0)$  gives rise to a unique solution, after the compatibility conditions  are satisfied. This  may give rise to a number of free parameters among the values of  $\phi_k(0)$.
   We will see that these compatibility conditions are  in fact determined by Theorem C.

\smallskip

The main part of the proof is thus to compute the compatibility  conditions and show they can be solved.  We consider the various cases that lead to the equations \eqref{Metricsmooth} and \eqref{Ricsmooth} in the smoothness conditions separately.

\subsection{Vectors tangent to the singular orbit.}
Let us fix a basis $X_i$ of $\fm$.
We know that it is possible to combine the components of the metric on $\fm$ so that some linear combination will be equal to one of the forms
$$\sum a_{ij}^k g(X_i,X_j)=\begin{cases}a_k+t^2\,\phi_k &\\t^{d_k}\, \phi_k &\end{cases}$$
where  $\phi_k$ is an even function,  $d_k\geq 1$ is an integer and $a_k\neq 0$. The same (except that $a_k$ is allowed to be $0$) holds for the Ricci tensor and the tensor $T$.  We consider the two cases separately
\subsubsection{The linear combination does not vanish at $t=0$.}
The smoothness conditions state that
$$\sum a_{ij}^k \Ric(X_i,X_j)=b_k+t^2\,\eta_k, \quad \text{ and }\quad  \sum a_{ij}^k T(X_i,X_j)=\rho_k,$$
for some smooth even functions $\eta_k,\rho_k$ and some constant $b_k\ne 0$.  We thus need to  solve  an equation of the form
$$\sum a_{ij}^k \Ric(X_i,X_j)=\psi_k-\frac 12 t^2\,\phi_k''=\rho_k$$
for $ \phi_k''$, where $\psi_k=\psi_k(\phi_i,\phi_i')$ is another even function obtained by plugging the expressions for $g_{ij}$ in terms of $\phi_i$ into the formula for the Ricci tensor.  According to our description of the smoothness,  $\psi_k$ does not contain any second derivatives. Since both $\psi_k$ and $\rho_k$ are even, we can solve this equation for $\phi_k''$ if and only if  
$$
\psi_k(0)=b_k=\rho_k(0)
$$
which is possible if and only if
$$\sum a_{ij}^k \Ric(X_i,X_j)(0)$$
can be arbitrarily assigned, even after the metric on the singular orbit and the second fundamental form are chosen.

For an Einstein metric the equation becomes  $\sum a_{ij}^k \Ric(X_i,X_j)(0)=\lambda a_k$ for some (arbitrarily chosen) Einstein constant $\lambda$.  Since $a_k$ is determined by the metric at the singular orbit, and is chosen arbitrarily ahead of time, the same proof solves the initial value problem for an Einstein metric near the singular orbit.

\subsubsection{The linear combination  vanishes at $t=0$.}\label{t=0}
Here we have
$$\sum a_{ij}^k \Ric(X_i,X_j)=t^{d_k}\,\eta_k,\quad \sum a_{ij}^k T(X_i,X_j)=t^{d_k}\,\rho_k.$$
and we need to solve 
$$\sum a_{ij}^k \Ric(X_i,X_j)=t^{d_k}\psi_k-\frac 12 t^{d_k}\,\phi_k''=t^{d_k}\,\rho_k$$
for $\phi_k''$, where $\psi_k$ is some new smooth even function. 
 Thus in this case there are no compatibility conditions.


\subsection{$X,Y$ in the slice $\fp$}
In this case there is no need to combine the components of the metric and the Ricci tensor since smoothness is equivalent to the assumption that the  components of the metric or $\Ric$ are even functions. Thus it is sufficient to consider any $X=Y$. 
Smoothness implies that
$$
g(X,X)=t^2+t^4\,\phi_X, \quad \Ric(X,X)=t^2\,\eta_X, \quad T(X,X)=t^2\,\rho_X
$$ 
for a unit vector $X\in\fp$, i.e., $|X|_{g_0}=1$.
Notice that for the Ricci curvature and for $T$, unlike for the metric, the second derivative at $t=0$ can be arbitrary.
 We need to solve 
$$\Ric(X,X)=t^2\,\eta_X=t^2\,\psi_X-\frac 12 t^4\,\phi_X''=t^2\,\rho_X$$
for $\phi_X''$, where $\psi_X$ is another even function.
Hence the compatibility condition  is satisfied 
 if and only if
$$\lim_{t\to 0} \frac 1 {t^2} \Ric(X,X)$$
can be arbitrarily assigned, even after the metric on the singular orbit and the second fundamental form are chosen.
 
\subsection{$X$ tangent to the singular orbit and $Y$ in the slice} Let $X_i$ be a basis of $\fm$, and $Y_j$ a basis of $\fp$.
We know that it is possible to combine the components of the metric so that some linear combination will be of the form
$$\sum a_{ij}^k g(X_i,Y_j)= t^{d_k}\, \phi_k $$
where  $\phi_k$ is an even function, and  $d_k\geq 2$  an integer. For the Ricci tensor and $T$ two cases may occur:
$$\sum a_{ij}^k \Ric(X_i,Y_j)=\begin{cases}t^{d_k-2}\, \eta_k  &\\t^{d_k}\, \eta_k &\end{cases},\quad \sum a_{ij}^k T(X_i,Y_j)=\begin{cases}t^{d_k-2}\, \rho_k  &\\t^{d_k}\, \rho_k &\end{cases}$$
since, unlike for the metric, $\Ric(X_i,Y_j)(0)$  and $T(X_i,Y_j)(0)$ do not have to vanish. We consider the two cases separately.

\subsubsection{Ricci and $g$ vanish with the same order at $t=0$.}
Using the expression of the Ricci tensor, and substituting the metric, we get
$$\sum a_{ij}^k \Ric(X_i,Y_j)=t^{d_k}\psi_k-\frac 12 t^{d_k}\,\phi_k''$$
where $\psi_k$ is a new smooth function of $t$. If $\rho_k$ is  a prescribed even function, we want to see if it is possible to solve
$$\sum a_{ij}^k \Ric(X_i,Y_j)=t^{d_k} \rho_k.$$
This means we have to solve
$$\psi_k-\frac 12  \,\phi_k''=\rho_k$$
and hence there is no  compatibility condition for this case.

\subsubsection{Ricci and $g$ vanish with different orders at $t=0$.}
Using the expression of the Ricci tensor and the metric,  and using the fact that we know the order of vanishing at $t=0$, we get
$$\sum a_{ij}^k \Ric(X_i,Y_j)=t^{d_k-2}\psi_k-\frac 12 t^{d_k}\,\phi_k''$$
where $\psi_k$ is a smooth even function of $t$. As before, the function $\psi_k$ does not contain second derivatives. If $\rho_k$ is  a prescribed even function, we want to see if it is possible to solve
$$\sum a_{ij}^k \Ric(X_i,Y_j)=t^{d_k-2} \rho_k.$$
Thus we have to solve
$$\psi_k-\frac 12 t^2 \,\phi_k''=\rho_k$$
and hence the compatibility condition is
$$\psi_k(0)=\rho_k(0)$$
which is possible iff
$$\frac 1{t^{d_k-2}}\,\sum a_{ij}^k \Ric(X_i,Y_j)(0)$$
can be arbitrarily assigned, even after the metric on the singular orbit and the second fundamental form are prescribed. This is equivalent to prescribing the first nonzero derivative of $\sum a_{ij}^k \Ric(X_i,Y_j)$ at $t=0$. This case is significantly more complicated and will be discussed in Part II.

\begin{rem}
\label{cnotrel}
Note that $Y_i$ vanishes at $t=0$, hence $\Ric_M(X_i,Y_i)(0)=0$. We will show, in the second part of the paper, that, writing
$$g_{|\fm\fp} =t^2 C_0+t^3 C(t)$$
the elements of $C_0$ are not involved in the compatibility conditions.  This implies that we can have compatibility conditions only if $d_k>2$. 
\end{rem}


\section{\bf Contribution from the second fundamental form}

We start by computing the terms in $\Ric_M(e_u,e_v)(0)$ that depend on $L$, again without assuming that $C=0$.   Recall that 
\begin{equation}
L= -\frac 14 \tr(P^{-1}P') P' +\frac 12 P' P^{-1} P'-\frac 12 P''.
\end{equation}
and that we chose $Q$ such that $Q_{|\fp}$ is the Euclidean inner product $g_0$.

Let 
\begin{equation}\label{Pseries}
	P_{|\fp}= t^2\, \Id+t^4 B,\qquad P_{|\fm}=A_0+t A_1+t^2 A,\qquad P_{|\fm\fp}=t^2 C_0+t^3 C
\end{equation}
where $A_0,A_1,C_0$ are constant matrices and  $A,B,C$ are smooth functions of  $t$.  Smoothness may imply further restriction on $A,B$ and $C$, but it turns out that  this is not needed for our computations.

For the inverse we have 
\begin{equation}\label{Pinverse}
	P^{-1}_{|\fp}= \frac{1}{t^2}-(B-C_0^tA_0^{-1}C_0) +\cdots  ,\quad P^{-1}_{|\fm}=A_0^{-1}-t \ A_0^{-1} A_1A_0^{-1} +\cdots,\quad P^{-1}_{|\fm\fp}=-A_0^{-1}C_0 +\cdots
\end{equation}
and it turns out that higher order terms will not be needed. 

\begin{prop}\label{extrinsic}
	The contribution to the compatibility conditions from  $L$  for the components of the Ricci tensor in $\fm$ or $\fp$, are as follows:
	\begin{eqnarray*}
		(a) \,L(\fm,\fm)(0)&=&- ( \dim(\fp)+1)\,A(0)-\frac 14 \tr (A_0^{-1}A_1)\cdot A_1 +\frac 12 A_1A_0^{-1}A_1+2C_0C_0^t.\\
		 (b)  \,\lim_{t\to 0}\frac1{t^2}L(\fp,\fp)&=&-2\dim(\fp)\ B(0)- \tr(B(0))\Id_\fp  + \frac12 \tr(A_0^{-1}A_1A_0^{-1}A_1)\ \Id_\fp\\
		&-&  \trace(A_0^{-1} A )\Id_\fp+  \tr(C_0^tA_0^{-1}C_0)\ \Id_\fp.
	\end{eqnarray*}
\end{prop}
\begin{proof}

	(a) $L(\fm,\fm)$.
	For the compatibility condition in $\Ric(\fm,\fm)$ we need to compute the constant term.
	The constant term in $-\frac 14 \tr(P^{-1}P') P'$  is easily seen to be equal to
	$$-\frac 14 \tr(A_0^{-1}A_1)\,A_1-  \dim(\fp)\,A(0)$$
	\smallskip
	and in $\frac 12 P' P^{-1} P'$ the constant term is:
	$$
	\frac 12 A_1 A_0^{-1}A_1+2C_0C_0^t
	$$

	Finally, in $-\frac 12 P''$ the constant term is  $-A(0).$
	Summing these, we get part (a).
	
	\bigskip
	
	(b) $L(\fp,\fp)$.
	For the compatibility condition in $\Ric(\fp,\fp)$ we need to compute the order two term. In $-\frac 14 \tr(P^{-1}P') P'$  it is 
	
	\begin{align*}
		&-2\dim(\fp)\ B- \tr(B)\Id_\fp  + \frac12 \tr(A_0^{-1}A_1A_0^{-1}A_1)\ \Id_\fp - \trace(A_0^{-1} A )\Id_\fp+\tr(C_0^tA_0^{-1}C_0)\ \Id_\fp
	\end{align*}
	
	\smallskip

	and in $\frac 12 P' P^{-1} P'$ it is $6B(0)$ and in  $-\frac 12 P''$ it is clearly $-6B(0)$. Thus (b) holds.
	
\end{proof}
 Note that $L$ may not be smooth at $t=0$ since a $t^{-1}$ term will appear in the formula.
 
The case of the second fundamental form for inner products between $\fp$ and $\fm$ will be discussed in part II.

\section{\bf Contributions from the intrinsic Ricci curvature}

In this Section we discuss the compatibility conditions for the inner products $g(\fm,\fm)$ and $g(\fp,\fp)$ coming from the curvature of $G/H$,  again without assuming that  $C=0$. 
We will see that eventually, these contributions all vanish, and thus Theorem A and C follow from \pref{extrinsic}.

\smallskip

Recall that for the intrinsic Ricci curvature we have:
\begin{equation}\label{RicciGH}
\Ric_{G/H}(e_u,e_v) =  -\frac12\sum_{r,s,p,q}\Gamma_{rp}^u\Gamma_{sq}^v P^{-1}_{rs} P_{pq}
+\frac14\sum_{r,s,k,l,p,q} \Gamma_{rk}^p\Gamma_{sl}^q  P^{-1}_{rs}P^{-1}_{kl}  P_{pu}P_{qv}
\end{equation}
Here we dropped $B$, which does not depend on the metric,  and $Z$ since they will not contribute to the compatibility conditions. 

\smallskip
Recall also that we chose an index sets $I$ and $J$ such that
$
e_u\in \fp \text{ if } u\in I, \quad \text{and} \quad e_u\in \fm \text{ if } u\in J
$.
Our goal is to compute the intrinsic contributions to the compatibility conditions by substituting  \eqref{Pseries} and \eqref{Pinverse}  into \eqref{RicciGH}.

\subsection{Inner products in $\fm$}
According to the discussion of the previous sections we only need to consider the case where 

\begin{equation}\label{sum}\sum a_{ij}^k g(X_i,X_j)=a_k+t^2\,\phi_k  \quad \text{and}\quad \sum a_{ij}^k \Ric(X_i,Xj)=b_k+t^2\,\eta_k
\end{equation}
where
  $\phi_k$ and  $\eta_k$ are even functions and $a_k\neq 0$. We then  must be able to prescribe the value of
$b_k$ arbitrarily. We first want to argue that it is sufficient to show that $\Ric(X_i,X_j)(0)$ can be prescribed arbitrarily for all $i,j$. For this, recall that we have a the decomposition $$\fm=\fm_0\oplus \fm_1\oplus \cdots \oplus \fm_r$$   into irreducible and trivial $K$ modules. 

The derivation of  the smoothness conditions in  \cite{VZ1} are obtained by choosing a circle $L\subset K$ and decomposing $\fm$ under the action of $L$ (see Section 1). This implies that in \eqref{sum} we only need to take linear combinations of components  of $\Ric(X_i,X_j)$ where $X_i,X_j$ both lie in the same module $\fm_k$ or $\fm_0$, or $X_i\in\fm_k, X_j\in \fm_\ell$  for some $k\ne \ell$. We can also assume that $\Ric(\fm_k,\fm_\ell)$ does not vanish identically, which may be forced by some smoothness condition, since otherwise there is no compatibility condition, Since ${\Ric_M}(0)_{|\fm}$ is $K$ invariant  we have the following cases:
\begin{itemize}
	\item If $X_i,X_j\in \fm_k$ for some $k$,  then, since $\Ric(0)_{|\fm_k}=c \Id$ for some constant $c$,  the sum, at $t=0$, consists of only one term. 
	\item If $X_i\in \fm_k$, $Y_j\in \fm_\ell$ and $\fm_k$ and $\fm_\ell$ are non-equivalent, then $\Ric(0)(\fm_k,\fm_\ell)=0$ and hence there is no condition.
	\item 
If $X_i\in \fm_k$,  $Y_j\in \fm_\ell$ and $\fm_k$, $\fm_\ell$ are equivalent and of real type, then
		there is only one non-zero entry, namely $\Ric(0)(X_i,Y_i)$ where $X_i\to Y_i$ is the unique (isometric) equivalence (see Section 1) and hence the sum consists again of only one term. 
	\item If the equivalence is of complex type, there are only  two non-zero entries $\Ric(0)(X_i,Y_i)$ and $\Ric(0)(X_i,JY_i)$ where $J$ is the complex structure. The sum consists of only two terms and the claim follows.
	\item  Similarly,  if the equivalence is of quaternionic  type, there are only four nonzero distinct entries.
	\item  For $\Ric(0)_{|\fm_0}$ we will see that we in fact need to prescribe all values of $\Ric(0)(X_i,X_j)$. 
\end{itemize}

Notice  that this also shows that the various cases all involve different entries in $\Ric(0)_{|\fm}$ and hence the compatibility can be solved independently.

\smallskip
 
  We can thus  discuss the various cases separately according to the irreducible and trivial $K$ modules. Notice also that on a $K$ irreducible module $A_0=a\Id$ for some $a>0$, whereas this is not the case for $A_1$ or $A(0)$.
We start with the following observation.
\begin{lem}
For $e_u, e_v\in \fm$, $\Ric_M(e_u,e_v)(0)$ does not depend on $B(0)$ and $C(0)$.
\end{lem}
\begin{proof}
Using \eqref{RicciGH} and \eqref{Pseries}, \eqref{Pinverse}, we get, for the intrinsic curvature, two terms involving $B(0)$:
$$
\frac12\sum_{r,s\in I,\; p\in J}\Gamma_{rp}^u\Gamma_{sp}^v (B_{rs}(0) (A_0)_{pp}) -\frac12\sum_{r,s\in I,\; p\in J}\Gamma_{rp}^u\Gamma_{sp}^v (B_{rs}(0)  (A_0)^{-1}_{pp} (A_0)_{uu} (A_0)_{vv})
$$
If $e_u,e_v$ lie in the same $K$ irreducible module, then, unless some structure constant vanishes,  $(A_0)_{pp}=(A_0)_{uu}=(A_0)_{vv}$ and the terms cancel. If $e_u,e_v$ lie in different $K$ irreducible modules, then either $\Gamma_{rp}^u$ or $\Gamma_{rq}^v$ vanish, and similarly for the second term. If $e_u,e_v$ lie in a trivial module, the structure constants vanish.
It is easy to see that $\Ric_{G/H}(e_u,e_v)(0)$  depends only on $C_0$ but  not  on $C(0)$, since $C(t)$ vanishes at order at least $3$.

As we saw in \pref{extrinsic}, $B(0)$ and $C(0)$ do not enter in the extrinsic curvature either, and hence the claim follows.
\end{proof}

Now we compute the contribution from $A(0)$ in the curvature of $G/H$.  Surprisingly, these contributions always cancel in $\Ric_M(0)$, but only after we add the extrinsic curvature and take the trace.

 Using \eqref{RicciGH}, one easily sees that contributions from $A(0)$ to $\Ric_{G/H}(e_u,e_v)$, 
writing $A$ for simplicity, consist of 4 terms:

\begin{align}\label{mRic1}
-&\frac12\sum_{r\in I,\; p,q\in J}\Gamma_{rp}^u\Gamma_{rq}^v A_{pq}\nonumber\\
+&2\cdot\frac14\sum_{r\in I ,\; k,l,p,q\in J} \Gamma_{rk}^p\Gamma_{rl}^q  \left(-A_0^{-1}AA_0^{-1}\right)_{kl}\left(A_0\right)_{pu}\left(A_0\right)_{qv}\\
+&2\cdot\frac14\sum_{r\in I ,\; k,l,p,q\in J} \Gamma_{rk}^p\Gamma_{rl}^q  \left(A_0^{-1}\right)_{kl}\left(A_0\right)_{qv}A_{pu}\nonumber\\ 
+&2\cdot\frac14\sum_{r\in I ,\; k,l,p,q\in J} \Gamma_{rk}^p\Gamma_{rl}^q  \left(A_0^{-1}\right)_{kl}\left(A_0\right)_{pu}A_{qv}\nonumber
\end{align}

We will in fact see that, in the cases involved in the compatibility conditions, the terms involving $A_0$ always cancel, and we get that the contribution of $A(0)$ to  the intrinsic curvature is given by

\begin{equation}
\label{mRic2}
-\sum_{r\in I,\; p,q\in J}\Gamma_{rp}^u\Gamma_{rq}^v A_{pq} +\frac12\sum_{r\in I ,\; p, k\in J} \Gamma_{rk}^p\Gamma_{rk}^v  A_{pu} +\frac12\sum_{r\in I ,\; p,k\in J} \Gamma_{rk}^p \Gamma_{rk}^u A_{pv}
\end{equation}

\subsubsection{\bf{Irreducible $K$ modules}}\label{irred}
As we saw, it is  sufficient to compute $\Ric_M(0)(X,X)$ for any $X\in\fm_k$ and show it can be prescribed arbitrarily.
 We  break up the index set $J=J_0\cup J_1\cup\cdots\cup J_k$ such that $e_u\in\fm_i$ if $u\in J_i$. Clearly $\Gamma_{rq}^u=0$ if $r\in I,\ u\in J_k, q\in J_l$ with $k\ne l$. It follows that, in \eqref{mRic1}, the contributions of $A_0$ cancel.  For $e_u\in \fm_i$  \eqref{mRic2} and  \eqref{skew} imply that the contribution of $A(0)$ in $\Ric_{M}(0)(e_u,e_u)$, adding the extrinsic contribution, is equal to:
$$
 -\sum_{r\in I,\; p,q\in J}\Gamma_{rp}^u\Gamma_{rq}^u A_{pq} +\sum_{r\in I ,\; p,k\in J} \Gamma_{rk}^p \Gamma_{rk}^u A_{pu} +L(e_u,e_u)(0)
$$
Since $\Ric_M$ is a multiple of the $\Id$ on $\fm_i$, we can  take the trace and divide by $\dim \fm_i$.  Using \eqref{skew}, and renaming indices, the intrinsic part vanishes since:
 \begin{align*}
 	&-\sum_{r\in I,\; p,q\in J,\; u\in J_i}\Gamma_{rp}^u\Gamma_{rq}^u A_{pq} +\sum_{r\in I ,\; p,k\in J,\; u\in J_i} \Gamma_{rk}^p \Gamma_{rk}^u A_{pu}=\\
 	& -\sum_{r\in I,\; k,p\in J,\; u\in J_i}\Gamma_{rp}^k\Gamma_{ru}^k A_{pu} +\sum_{r\in I ,\; p,k\in J,\; u\in J_i} \Gamma_{rp}^k \Gamma_{ru}^k A_{pu}=0
 \end{align*} 
Thus
$$
\Ric_M(0)_{|\fm_i}=\frac1{\dim \fm_i}\tr(L(0))_{|\fm_i} +f(A_0,A_1,C_0)
$$
 and according to \pref{extrinsic} (a) we can solve the compatibility condition for $\tr A(0)_{\fm_i}$ since its coefficient is non-zero.
 
   We note that  only $\Ric_M(0)_{|\fm}$ is $K$ invariant, but not $\Ric_{G/H}(0)_{|\fm}$  or $L(0)_{|\fm}$, and thus the cancellations will only be visible after we take the trace in $\fm_i$. The contributions of   $A_0, A_1$ and $C_0$ in $f$ are explicit in the extrinsic part according to  \pref{extrinsic},  but there are further ones in $\Ric_{G/H}(0)$,   and a constant term as well. We leave it to the reader to compute those explicitly from \eqref{RicciGH}.

\subsubsection{\bf{Equivalent $K$ Modules}}
We now consider the case of two equivalent $\Ad_K$ irreducible modules  $\fm_1, \fm_2$.     Let $J_k, k=1,2$ be the index sets such that $e_u\in\fm_k$ for $u\in J_k$.  
 We first assume that the representations are of real type. Thus there exists a unique isometry $f\colon \fm_1\to\fm_2$, an intertwining operator between $\fm_1$ and $\fm_2$, i.e.,  $f([v,w])=[v,f(w)],\ v\in \fk, w\in\fm_1$. We choose a $Q$-orthonormal basis $e_i$  of $\fm_1$, and the $Q$-orthonormal basis
 $\e_{\oi}=f(e_i)$   of $\fm_2$. Notice that this also induces a bijective map on index sets: $F\colon J_1\to J_2$ with $F(k)=\bar k$. For the structure constants we  get the relations:
\begin{equation}\label{Grel}
\text{ for } r\in I \text{ we have } \Gamma_{ri}^k=\Gamma_{r\bar{i}}^{\bar{k}} \text{ and } \Gamma_{r\bi}^k=0 \text{ for all }  i,k\in J_1
\end{equation}
since, e.g., for $v\in\fk$ and $w_1,w_2\in\fm_i$
$$
\pro{[v,w_1]}{w_2}= \pro{f[v,w_1]}{f(w_2)}=\pro{[v,f(w_1)]}{f(w_2)}
$$
By Schur's Lemma and $K$ invariance of  $\Ric_M(0)$ we have
\begin{equation*}
\Ric_M(0)(e_{u},e_{\bar v})=0,\quad \Ric_M(0)(e_{u},e_{\bar u})=\Ric_M(0)(e_{v},e_{\bar v})\text{ for } u,v\in J_1,\  u\ne v
\end{equation*}
  and we  need to show that  $\Ric_M(0)(e_{u},e_{\bar u})$ can be prescribed arbitrarily.

Although $A_0$ is a multiple of $\Id$ on $\fm_1$ and $\fm_2$, they may not be the same. The contributions of $A_0$ nevertheless cancel due to 
\eqref{Grel}. We can then use \eqref{mRic2} and \eqref{Grel} and, since $\Gamma_{rq}^{p}=0$ if $e_p,e_q$ lie in different $K$ irreducible modules,   one obtains the contribution of $A(0)$ in $\Ric_{G/H}(0)(e_u,e_{\bar u})$:

\begin{align}\label{mRic3}
	& -\sum_{r\in I,\; p\in J_1, q\in J_2}\Gamma_{rp}^u\Gamma_{rq}^{\bar u} A_{pq} +\frac12\sum_{r\in I ,\; p, k\in J_2} \Gamma_{rk}^p\Gamma_{rk}^{\bar u}  A_{pu} +\frac12\sum_{r\in I ,\; p,k\in J_1} \Gamma_{rk}^p \Gamma_{rk}^u A_{p\bar u}\nonumber\\
	&=-\sum_{r\in I,\; p,q\in J_1}\Gamma_{rp}^u\Gamma_{r\bar q}^{\bar u} A_{p\bar q} +\frac12\sum_{r\in I ,\; p, k\in J_1} \Gamma_{r\bar k}^{
		\bar p}\Gamma_{r\bar k}^{\bar u}  A_{\bar p u}+\frac12\sum_{r\in I ,\; p,k\in J_1} \Gamma_{rk}^p \Gamma_{rk}^u A_{p\bar u}\\
	&=-\sum_{r\in I,\; p,q\in J_1}\Gamma_{rp}^u\Gamma_{r q}^{ u} A_{p\bar q} +\frac12\sum_{r\in I ,\; p, k\in J_1} \Gamma_{r k}^{
		 p}\Gamma_{r k}^{ u}  A_{u \bar p }+\frac12\sum_{r\in I ,\; p,k\in J_1} \Gamma_{rk}^p \Gamma_{rk}^u A_{p\bar u}\nonumber
\end{align}
Since $\Ric_M(e_u,e_{\bar u})$ is independent of $u$, after adding the extrinsic contribution,  we can  take the trace in $u$ and divide by $\dim \fm_i$.
The intrinsic part then vanishes since
$$
-\sum_{r\in I,\; p,q,u\in J_1}\Gamma_{ru}^p\Gamma_{r u}^{ q} A_{p\bar q} +\frac12\sum_{r\in I ,\; p, u\in J_1} \Gamma_{r u}^{
	p}\Gamma_{r u}^{ q}  A_{q \bar p }+\frac12\sum_{r\in I ,\; p,u\in J_1} \Gamma_{ru}^p \Gamma_{ru}^q A_{p\bar q} =0
$$
 Thus
$$
 \Ric_M(0)(e_u, e_{\bar u})=\frac1{\dim \fm_i} \sum_{u\in J_1} L(0)(e_u, e_{\bar u}) +f(A_0,A_1,C_0)
 $$
 and \pref{extrinsic} (b) implies that  we can solve the compatibility condition for any of the  variables $A_{u \bar u}$ with $u\in J_1$.
 
\smallskip

Next we consider the case where $\fm_i$ are of complex type. Thus there exists a complex structure $J\colon \fm_1\to\fm_1$ and we can choose a compatible $Q$ orthonormal basis  $\{e_u,J(e_u)\}$ of $\fm_1$ and the corresponding $Q$ orthonormal basis $\{e_{\bar u},J(e_{\bar u})\}$ of $\fm_2$. Due to Schur's Lemma we have  only two non-vanishing components for $\Ric_M(0)(\fm_1,\fm_2)$:
\begin{equation}\label{Ricrel}
\Ric_M(0)(e_u,e_{\bar u})=\Ric_M(0)(J(e_u),J(e_{\bar u})),  \quad \Ric_M(0)(e_u,J(e_{\bar u})) =- \Ric_M(0)(e_{\bar u},J(e_{ u}))
\end{equation}
and we need to show that they  can be prescribed arbitrarily.
 The formulas are actually the same as in the previous case by using the two intertwining operators $$f_1(e_u)=e_{\bar u},\ f_1(J(e_u))=J(e_{\bar u}),\quad \text{ and } \quad f_2(e_u)=J(e_{\bar u}),\ f_2(J(e_u))=e_{\bar u}$$
together with the two bijections $F_1,F_2\colon J_1\to J_2$ of the index sets. Similarly, if $\fm_1$ is of symplectic type, we have three complex structures on $\fm_1$ and obtain the same conclusion.

\smallskip

\subsubsection{\bf{Trivial $K$ Module}}
If $e_u, e_v\in\fm_0$, then all structure constant vanish since $[e_r,e_q]=0$ when $r\in I$ and $q\in\fm_0$ and there is no intrinsic contribution. According to \pref{extrinsic},  all entries $A_{uv}(0)$ are thus determined by the compatibility conditions. Since a trivial $K$ module is also a trivial $H$ module, smoothness implies that $A_{uv}(t)$ are even functions, and that  $A_{uv}(0)$ is not forced to be $0$. Thus we always have compatibility conditions.

\subsection{Inner products in  $\fp$}
In this case smoothness simply means that all components of $\Ric$ are even functions. Thus  we have $\Ric_M(e_u,e_u)=t^2\phi_u$ for some even function $\phi_u$.   We want to compute $\lim_{t\to 0}\frac1{t^2}\Ric(e_u,e_u)$ and show that it can be prescribed arbitrarily. Since the limit is also $K$ invariant, it is independent of $u\in I$.  
From \eqref{RicciGH} it follows that the contribution from $B(0)$ (that we will denote by $B$ for simplicity) in the first part of the  intrinsic curvature is equal to
$$
-\frac12\sum_{r,p \in I}\Gamma_{ru}^p\Gamma_{su}^p (-B_{rs})-\frac12\sum_{r,p \in I}\Gamma_{ru}^p\Gamma_{ru}^q B_{pq}=
\frac12\sum_{r,p \in I}\Gamma_{rp}^u\Gamma_{sp}^u B_{rs}-\frac12\sum_{r,p \in I}\Gamma_{pr}^u\Gamma_{ps}^u B_{rs}=0
$$
In the second part we get two terms:
$$
\frac14\sum_{r,k,p \in I}\Gamma_{rk}^p \Gamma_{rk}^u B_{pu}-\frac14\sum_{r,k \in I}\Gamma_{rk}^u \Gamma_{sk}^u B_{rs}
$$
Adding the  extrinsic part in \pref{extrinsic} (b), we can take the trace over $u$, and we see that the above intrinsic contribution vanishes.
Thus
 \begin{equation*}
 	\lim_{t\to 0}\frac1{t^2}\Ric_M(e_u,e_u)= \frac1{\dim \fm_0}\sum_{u\in I} L(0)(e_u, e_{ u})=-3  \tr B(0)+f(A_0,A_1,A(0),C_0).
 \end{equation*} 
and we can solve for $\tr B(0)$.
 \bigskip

 We summarize the conclusions:
 \begin{itemize}
 	\item There is only one compatibility condition  for each  irreducible $\Ad_K$ module $\fm_i$, and we can solve for  $\tr A(0)_{|\fm_i}$.
 	\item When $Ric_M(\fm_i,\fm_j)(0)$  is not forced to be $0$ by the smoothness conditions, there are one, two or four compatibility conditions  when $\fm_i$, $\fm_j$ are two equivalent irreducible  modules, depending on whether the representation is real, complex, or symplectic. We can solve them for any entry in $A(0)(\fm_i,\fm_j)$.
 	\item $A(0)_{|\fm_0}$ is completely determined by the compatibility conditions.
 	\item There is only one compatibility condition in $\fp$.
 \end{itemize}
Notice though that the condition for $\fp$ involves $A(0)$,  hence we first determine the entries for $A(0)$, and then solve foe  $\tr B(0)$. Furthermore,  the various components of $A(0)_\fm$  are  involved in only one of the compatibility conditions. Hence all compatibility conditions can be satisfied.

\smallskip

This finishes the proofs of Theorems A and B in the Introduction. Notice that we proved Theorem C along the way.

\section{\bf Examples}\label{examples}
We now illustrate the methods in a few examples.

\subsection{Example 1} 
Consider the cohomogeneity one manifold defined by the triple $SO(2)\subset SO(2)\times SO(2)\subset SO(4)$, where the $SO(2)$ are diagonal blocks. Let
$$Z=E_{34},\quad V_1=E_{13}+E_{24},\quad V_2=-E_{14}+E_{23},\quad V_3=E_{14}+E_{23},\quad V_4=-E_{13}+E_{24}.$$
We have
$$\fp=\fp_0=\{Z\},\quad \fm_1=\{V_1,V_2\},\quad \fm_2=\{V_3,V_4\}.$$
Notice that the slice is orthogonal to the orbits, i.e. $C=0$, and hence our theorems apply. The modules $\fm_1$ and $\fm_2$ are irreducible complex equivalent $H$-modules, which are also irreducible $K$ modules, but are non-equivalent under the action of $K$.   
For the metric we have 
$$f=\ml Z,Z\mr, \qquad g_1=\ml V_1,V_1\mr =\ml V_2,V_2\mr,\qquad g_2=\ml V_3,V_3\mr =\ml V_4,V_4\mr,$$
$$
b_1=\ml V_1,V_3\mr =\ml V_2,V_4\mr,\qquad b_2=\ml V_1,V_4\mr=-\ml V_2,V_3\mr
$$
According to Lemma 5 and  6 in \cite{VZ1}, since $d_1=d_2=1$ and $a=1$, smoothness  implies that
$$
f=t^2\Id +t^4 \phi_1,\qquad  g_1=a_1+t^2\phi_2,\qquad  g_2=a_2+t^2\phi_3
$$
$$
b_1=t^2\phi_4,\qquad b_2=t^2\phi_5.
$$
Thus we have
$$
B=\phi_1\Id,\qquad A_0=\diag ( a_1\Id, a_2\Id)\qquad A_1=0
$$
and
$$ A=\begin{pmatrix}
	\phi_2 &0 &\phi_4 &\phi_5\\
	0 &	\phi_2 &-\phi_5 &\phi_4\\
	\phi_4 &-\phi_5 &\phi_3 &0\\
	\phi_5 &\phi_4 &0 &\phi_3\end{pmatrix}$$
Thus the singular orbit is totally geodesic. Since $\fm_i$ are $K$ invariant and irreducible, the compatibility conditions determines $\phi_2(0)$ and $ \phi_3(0)$. Since they are in-equivalent as $K$ modules, there  is no compatibility condition on $\phi_4(0)$ and $ \phi_5(0)$. The value of $\phi_1(0)$ is  determined by the compatibility condition.
Thus we can prescribe the tangential part of the Ricci tensor (or the Einstein equations), and after choosing the metric of the singular orbit, i.e., $a_1,a_2$, the values of $\phi_4(0), \phi_5(0)$ are free parameters.

In this example it is not possible to fully prescribe $Ric_M(\dot c,\dot c)$. Indeed, \cite{VZ1} equation (5) shows  that smoothness implies that $Ric_M(\dot c,\dot c)=\psi_1(t)$,  $Ric_M(Z,Z)=t^2\psi_2(t)$ and $Ric_M(Z,\dot c)=t\psi_3$ where $\psi_1$, $\psi_2$ and $\psi_3$  are even functions. Furthermore, it also implies that $\psi_1(0)=\psi_2(0)$. 

On the other hand, we will see in  Section 6.2  that we can prescribe  $Ric_M(\dot c,\dot c)$ if we allow to reparametrize the geodesic $c$.

\subsection{Example 2}
 Consider the cohomogeneity one manifold defined by the triple $SO(2)\subset SO(3)\subset SO(4)$, where all the embeddings are (upper) block diagonal. Thus $\fh=\{E_{12}\}$. Let
$$V_1=E_{13},\quad V_2=E_{23},\quad V_3=E_{14},\quad V_4=E_{24},\quad Z=E_{34}.$$
Under the $\Ad_H$ decomposition  we have
$$\fp=\{V_1,V_2\},\quad \fm_0=\{Z\},\quad \fm_1=\{V_3,V_4\}.$$
Here $\fp$ and $\fm_1$ are complex equivalent $H$-modules and $\fm=\fm_0\oplus \fm_1$ is an irreducible $K$-module.
The metric is defined by the functions:
$$f=\ml V_1,V_1\mr =\ml V_2,V_2\mr ,\qquad g_1  =\ml Z,Z\mr , \qquad  g_2=\ml V_3,V_3\mr =\ml V_4,V_4\mr,$$
$$
b_1=\ml V_1,V_3\mr =\ml V_2,V_4\mr,\qquad b_2=\ml V_1,V_4\mr=-\ml V_2,V_3\mr
$$
Using Lemma 5, Lemma 7(a) and Lemma 9  in \cite{VZ1},   smoothness implies that
$$
f=t^2\Id +t^4 \phi_1,\qquad  g_1=a_0+t^2\phi_2,\qquad  g_2=a_0+t^2\phi_3
$$
$$
b_1=t^2\phi_4,\qquad b_2=t^3\phi_5
$$
 and thus $C_0=0$. But the order of vanishing for the Ricci tensor, according to Table D in \cite{VZ1}, is  two less for $b_1$ and $b_2$ since $g(\dot c,V_i)=0$. If we assume that $C=0$, i.e., $b_1=b_2=0$, a computation shows that then  $\Ric_M(\fp,\fm)=0$ as well, and we may apply the results of the previous sections for such metrics. In this case we have:
$$
B=\phi_1\Id,\qquad A_0=a_0\Id,\qquad A_1=0,\qquad A=\diag(\phi_2,\phi_3,\phi_3)
$$
 The compatibility conditions in  Theorem C imply that, after we choose $a_0$, the values $\phi_1(0)$ and $\tr A(0)=\phi_2(0)+2\phi_3(0)$ are  determined, and we thus have one free parameters, say $\phi_2(0)=g_1''(0)$, in addition to the choice of $a_0$. In Part II we will discuss the case when $C\ne 0$.
 
 Note that now $\fm_0$ is not necessarily orthogonal to $\dot c$. In this example  smoothness implies that $Ric_M(\dot c,\dot c)=\psi_1$,  $\Ric_M(Z,Z)= \psi_2$  and  $Ric_M(Z,\dot c)=t\psi_3$ where $\psi_1$, $\psi_2$ and $\psi_3$  are even functions.   But now there is no relation between $\psi_1(0)$ and $\psi_3(0)$.
 \smallskip 
 
 There are many examples of this type,  see \cite{KPS}, where $C=0$ implies  $\Ric_M(\fp,\fm)=0$, and hence our Theorems apply.

\subsection{Example 3}
The results have an interesting application to Einstein metrics on Euclidean space. For each transitive action of $K$ on a sphere, we get a corresponding family of cohomogeneity one metrics on Euclidean space. The case with most freedom is  the  action of $K=Sp(n)$  on $\QH^n=\R^{4n}$. It give rise to a 7 parameter family of $K$ invariant cohomogeneity one metrics on $\R^{4n}$. These can be defined in terms of the Hopf fibration $\Sph^3\to \Sph^{4n-1}\to\QH P^{n-1}$. The isotropy representation of $H=Sp(n-1)$ on $\fn=\fp$ decomposes as $\fp_0\oplus\fp_1$ where $H$ acts trivially on $\fp_0\simeq\Im\QH$ and via matrix multiplication on $\fp_1\simeq\QH^{n-1}$. If $X_1,X_2,X_3$ is the standard basis in $\fp_0$, orthonormal in the Euclidean metric, and $Y_i$ a $Q$-orthonormal basis in $\fp_1$, then the metric is determined by the functions
$$
f_{ij}=\ml X_i,X_j\mr,\quad h=\ml Y_i,Y_i\mr
$$
and the smoothness conditions are 
$$
f_{ij}=t^2\,\delta_{ij} +t^4\phi_{ij},\quad h=t^2+t^4\psi
$$
for some even functions $\phi_{ij},\psi$. According to our results, there is only one compatibility condition, which in the case of Einstein metrics (with Einstein constant $\lambda$) says
\begin{equation}\label{free}
	\tr B(0)=-3(\Sigma_i\, \phi_{ii}(0)+4n\psi(0))=\lambda
\end{equation}
Thus we have $6$ free parameters $\phi_{ij}(0)$, and $\psi(0)$ is determined, which gives rise to a $6$-parameter family of analytic Einstein metrics in a neighborhood of the origin. 
 The case of $G=Sp(n)Sp(1)$ was studied in \cite{Wi} and $G=Sp(n)U(1)$  in \cite{Ch1,Ch2} in order to construct complete Einstein metrics and Ricci solitons. In the first  case $f_{11}=f_{22}=f_{33}$, in the second $f_{22}=f_{33}$, and in both cases $f_{ij}=0$ for $i\ne j$. Thus there exists a 2-parameter, resp. 3 parameter family of $G$ invariant metrics with 1 resp. 2 free parameters.

\section{\bf Further solutions}
\bigskip

We now indicate that one can apply the same proof with minor changes to obtain solutions to other geometric problems.

\subsection{Solitons}
Recall that $(M,g,u)$ is called a $m$-quasi Einstein metric if 
\begin{equation}\label{MQ}
\Ric_{u,m}(g)=\Ric(g) +\Hess(u)-\frac1{m} du\otimes du = \lambda g
\end{equation}
where $\lambda$ is the quasi Einstein constant. When $\frac1{m}=0$  solutions are called gradient Ricci solitons.
On a cohomogeneity one manifold  we assume that $g$ and $u$ is $G$-invariant. Notice that if $G$ is compact, the $G$-invariance of $u$ follows via integration, although in the result below, we do not assume that $G$ is compact.   

In \cite{DW2,Wi} it was  shown that, if there exists  singular orbit,   the second Bianchi identity implies that $\Ric(X,\dot c)=0$ for $X\in \fp\oplus\fm$. Furthermore,  on a  cohomogeneity one manifold  \eqref{MQ} becomes (using $L=\frac12 P^{-1}P'$ in their notation):
$$
\Ric_M(x,y)=\Ric_{G/H}  -\frac 14 \tr(P^{-1}P') P' +\frac 12 P' P^{-1} P'-\frac 12 P'' + P'\, u' = \lambda,
$$
$$
\Ric_M(\dot c, \dot c)= \frac 14 \tr(P^{-1}P' P^{-1} P')-\frac12 \tr(P^{-1} P'')+u''-\frac1{m}(u')^2=\lambda
$$
  In \cite{DW2,Wi} it was also  shown, again using the second Bianchi identity,   that the second equation, setting $v=u'$, can be replaced by:
  
  $$
  v''+\frac12\tr(P^{-1}P') v'+\frac12 \tr(P^{-1}P' P^{-1} P')v-\frac12 \tr(P^{-1}P')v-2vv'-\frac1{m}v^2=\lambda
  $$

 The smoothness condition  for $g_{ij}$ is the same as before, and for $u$ it simply says that $u$ is even.  Notice also that we can assume that $u(0)=0$.   We again express the metric  in terms of the even functions $\phi_{ij}$, and  set   $v=2t\psi(t)$   for some even function $\psi$. One easily sees that the term $P'v$ in the first equation does not contribute to any of the compatibility conditions since  $v(0)=0$. Thus our methods imply that the first equation can be solved for $\phi_{ij}''$ in terms of $\phi_{ij}$, $\phi_{ij}'$ and $\psi$, and the second equation can be solved for $\psi''$.  We now apply Malgrange to the combined system, to obtain a convergent power series for the metric and potential.  Hence we also get the same number of free parameters in the metric.  This proves Theorem B in the Introduction.

In \cite{Buz,Wi} it was observed that the initial value problem can be solved as in \cite{EW}, again under  the assumption that $G$ is compact, $C=0$ and  that there is no trivial $\Ad_K$ modules in $\fm$.

\bigskip

\subsection{Prescribing $\mathbf{Ric(\dot c, \dot c)}$}

We now discuss that we can also prescribe the tensor $T$  along the geodesic, i.e., we also want to solve  $T(\dot c,\dot c)=\beta$ together with the tangential components. For this we  change the parametrization of the geodesic (see e.g. \cite{Pu2,Bu1} for special cases). Let $c(t)$ be parameterized by arc length and  $t=f(r)$ with $f'>0$, a change in parameter, and thus the metric is $dt^2+h_t=h^2dr^2+g_r$ with $h=f'$. We can now use our formulas for the curvature in arc length, using the conversion
$$
 \dot P=\frac{dP}{dr} \text{ and hence } P' =  \frac1{h}\ \dot{P},\ P''= - \frac{\dot{h}}{h^3}\dot P+\frac{1}{h^2}
P^{..}
$$
Using \eqref{shape} and \eqref{Riccic} we obtain a  system of equations:
\begin{equation}\label{tangential}
(\Ric_M)_{|G/H} =\Ric_{G/H}+\frac1{4h^2} (- \tr(P^{-1}\dot P) \dot P + \dot P P^{-1} \dot P - \overset{\cdot\cdot}{P} +\frac{2\dot{h}}{h}\dot P)
=T_{|G/H}
\end{equation}
\begin{equation}\label{geodesic}
\Ric_M(\dot c,\dot c)=  \frac{\dot{h}}{h^3}\tr(P^{-1}\dot P) +
\frac{1}{4h}\tr(P^{-1}\dot PP^{-1}\dot P) -\frac{1}{2h}\tr(P^{-1}\overset{\cdot\cdot}{P})=\beta
\end{equation}
where $T(\dot c,\dot c)=\beta$.

We now replace the second equation \eqref{geodesic} by an ODE for $h$ which does not involve the second derivative of $P$, using again the second Bianchi identity $\Div\Ric_M(\dot c)=2d(\Scal_M)(\dot c)$.
A straight forward computation shows that this equation in terms of $r$ is as follows. Here we choose a basis $e_0=\dot c, e_1,\dots,e_n$ where $e_1,\dots,e_n$ is a  $Q$-orthonormal basis of $\fp\oplus\fm$. All indices go from $1$ to $n$. We then have:
\begin{align}\label{replace}
&\frac1{2h^2}\dot\beta-\frac{\dot h}{h^3}\beta=\frac12\tr\left(P^{-1}\dot T\right)-\sum_{uv} P^{-1}_{uv}e_u(T_{v0})
-\frac12\sum\Gamma_{uv}^iP^{-1}_{uv}T_{i0}\\
& +\sum_{s,m,i} \Gamma_{ms}^s(P^{-1}T)_{m0 }
-\frac1{2h^2}\tr(P^{-1}\dot P)\beta\nonumber
\end{align}

The formulas for $\Ric(e_u,e_v)_{G/H}$ are of course the same in terms of $P(r)$ and so are the smoothness conditions for $g_{ij}(r)$, and we express $g_{ij}$ in terms of even functions $\phi_{ij}(r)$. Smoothness for $h$ means that $h$ is even and we set $h=h_0+t^2\psi(r)$ for some positive constant $h_0$ and an even function $\psi$.

We can now  solve the system \eqref{tangential}, \eqref{replace} in terms of $(g_{ij}(r),h(r))$ by solving for $\overset{\cdot\cdot}{\phi}_{ij}$ and $\overset{\cdot}{\psi}$.
 The term $\frac{h'}{2h^3}\dot P$ in \eqref{tangential}  does not contribute to the compatibility conditions in Section 4 since $\dot h(0)=0$.
Hence we get a formal power series for $g$ and $h$ and we can  apply Malgrange to obtain a solution Notice that the  initial condition $h_0>0$ can be chosen arbitrarily.

\bigskip

{\bf Remark} On the other hand, $\Ric(X,\dot c)$, with $X$ tangent to the orbits, can in general not be prescribed since the metric tensor $P$ tangent to the regular orbits, and the parametrization of $c$, are already determined, up to a finite number of constants. Of course if $\fp_0=\fm_0=\{0\}$, this term is automatically $0$. The only other freedom one has is the choice of the transverse curve $c$, which can help in special case. We also remark that if a tensor $T$  is diagonal in some $\Ad_H$ decomposition, then $\Ric$ is diagonal as well (see \cite{GZ}), and $\Ric(\fp,\fm)=0$. Thus our methods imply that the tensor $T$ can be prescribed. 

\subsection{Non-unimodular groups $G$}

We now indicate that the extra term in the intrinsic curvature of the principal orbit $G/H$ for a non-unimodular Lie group $G$ does not contribute to the compatibility conditions.

Recall that the extra term  that contributes to $\Ric(e_u,e_v)$ is
$$
-g( [Z,e_u],e_v) \quad \text{ with } Z=\sum_i U(X_i,X_i)
$$

A straightforward computation shows that this term is equal to:
\begin{equation}
	Z
	=\sum_{k,\ell}\Gamma_{ks}^sP^{-1}_{k\ell}  e_\ell\quad \text{ and } \quad
	-g( [Z,e_u],e_v)=- \sum_{k,\ell,s,t}\Gamma_{ks}^s\Gamma_{\ell u}^tP^{-1}_{k\ell}P_{tv}
\end{equation} 
Notice that if $k\in I$, t the expression  vanishes. This easily implies that in \eqref{RicciGH} the term does not contribute to $B(0)$ or $A(0)$, and thus there is no contribution to the compatibility conditions in $\Ric(\fm,\fm)$. By the same argument there is no contribution to $\Ric(\fp,\fp)$. It also does not contribute to $\Ric(\fp,\fm)$ since we need to lower the degree by two. Thus we can ignore the term $-g( [Z,e_u],e_v)$ in our proof of Theorems A and B.

\providecommand{\bysame}{\leavevmode\hbox
	to3em{\hrulefill}\thinspace}

\end{document}